\documentclass[11pt,twoside]{amsart}
\usepackage{amsmath}
\usepackage{tikz-cd}
\usepackage{amssymb}
\usepackage{amsthm}
\usepackage{latexsym}
\usepackage{amscd}
\usepackage{xypic}
\xyoption{curve}
\usepackage{ifthen}
\usepackage{hyperref}
\usepackage{graphicx}
\usepackage{float}
\usepackage{enumerate}

 \def\cocoa{{\hbox{\rm C\kern-.13em o\kern-.07em C\kern-.13em o\kern-.15em A}}}

\usepackage{xcolor}

\newtheorem{theorem}{Theorem}[section]
\newtheorem{lemma}[theorem]{Lemma}
\newtheorem{proposition}[theorem]{Proposition}
\newtheorem{corollary}[theorem]{Corollary}

\theoremstyle{definition}
\newtheorem{remark}[theorem]{Remark}
\newtheorem{definition}[theorem]{Definition}

\newtheorem{notation}[theorem]{Notation}

\newcommand {\Hom}{\mathrm{Hom}}
\newcommand {\Aut}{\mathrm{Aut}}

\newcommand {\Closure}{\mathrm{Closure}}

\newcommand {\ext}{\mathrm{ext}}

\newcommand {\Hilb}{\mathcal{H}\kern -0.25ex{\mathit ilb\/}}

\usepackage[none]{hyphenat}

\newcommand{\Pic}{\operatorname{Pic}}

\newcommand{\Num}{\operatorname{Num}}
\newcommand{\Div}{\operatorname{Div}}

\newcommand{\h}{\operatorname{h}}
\newcommand{\Ho}{\operatorname{H}}

\newcommand{\rank}{\operatorname{rank}}

\begin{document}

\author[Laura Costa]{Laura Costa}
\address{Department de matem\`{a}tiques i Inform\`{a}tica, Universitat de Barcelona, Gran Via de les Corts Catalanes 585, 08007 Barcelona,
Spain}
\email{costa@ub.edu}

\author[Irene Macías Tarrío]{Irene Macias Tarrio}
\address{Department de matem\`{a}tiques i Inform\`{a}tica, Universitat de Barcelona, Gran Via de les Corts Catalanes 585, 08007 Barcelona,
Spain}
\email{irene.macias@ub.edu}

\title[Moduli spaces on ruled 3-folds]{Moduli spaces of stable bundles on ruled 3-folds and Brill-Noether problems}

\begin{abstract}
In this paper, we redefine the theory of walls and chambers due to Qin developing a new tool to study moduli spaces of stable rank 2 vector bundles on algebraic varieties of higher dimension. We apply it to describe components of some moduli spaces of rank 2 stable bundles on ruled 3-folds as well as to prove that some Brill-Noether loci are non-empty.  
\end{abstract}
\thanks{Acknowledgements:   The authors  are partially   supported
by PID2020-113674GB-I00.}

\subjclass[2020]{MSC2020 code 14J60, 14J30}
\keywords{stable vector bundles, moduli spaces, walls and chambers, Brill-Noether}

\maketitle

\tableofcontents

\markboth{}{}

\section{Introduction}
The study of moduli spaces of stable vector bundles on algebraic varieties has long been a central theme in algebraic geometry, with profound connections to differential geometry, representation theory, and mathematical physics. Given $X$ a smooth projective irreducible variety of dimension $n\geq 1$ defined over an algebraically closed field $K$ of characteristic $0$ and given $H$ an ample divisor on $X$, we denote by $M_{X,H}(r;c_1, \dots, c_{\min \{r,n \}}) $ the moduli space of rank  $r$, $H$-stable  vector bundles $E$ on $X$ with Chern classes $c_i(E)=c_i$. These moduli spaces were constructed in the seventies by Maruyama. Since then,  many detailed and interesting results
have been proved regarding these moduli spaces when the underlying variety is a curve or 
a surface and very little is known if the underlying variety has dimension greater or equal
than three. Recent developments have underscored the rich and intricate structure of moduli spaces on higher-dimensional varieties. The transition from surfaces to 3-folds introduces a host of new geometric and analytical challenges, but also opens the door to deeper insights and applications. In addition, the study of moduli spaces of stable bundles on 3-folds is valuable since it advances our understanding of higher-dimensional algebraic geometry.

In contrast with moduli spaces of stable vector bundles on surfaces, on threefolds, moduli spaces are often non-reduced,  non-smooth and in general there are no general results about these moduli spaces concerning the number of connected components, dimension, smoothness, rationality,
topological invariants, etc... , making it harder to describe them geometrically. In fact, increasing the dimension of the base variety makes everything harder and to have explicit constructions is much more difficult. In addition,  they have intricate wall-crossing behavior (how the space changes as the stability condition varies), but studying these changes could help us understand birational geometry, deformation theory, and how different moduli spaces are connected. 
 Moduli spaces of stable bundles encapsulate deep geometric information about the underlying variety, understanding their structure is therefore crucial not only for the intrinsic study of vector bundles, but also for advancing broader programs. 

This paper aims to contribute to the study of moduli spaces of stable bundles on certain 3-folds. On one hand highlighting their dependence on the ample divisor (wall crossing theory) and on the other hand describing some Brill-Noether loci inside them, providing new results on Brill-Noether theory in higher dimension. Brill-Noether theory was introduced in the 19th century in the context of line bundles on curves. Later on, was generalized to rank $r$ stable vector bundles on curves and more recently, the theory has been stated for rank $r$ stable bundles on algebraic varieties of any dimension (see for instance \cite{Brill-Noether} and \cite{higherBN}). In spite that the existence of the Brill Noether locus $W_{X,H}(r;c_1, \dots, c_{\min \{r,n \}}) \subset M_{X,H}(r;c_1, \dots, c_{\min \{r,n \}}) $ that parametrizes $H$-stable rank $r$ vector bundles $E$ with at least $k$ independent sections is seated, very few examples are known when dealing with bundles on varieties of dimension equal o greater than 2 and there are a lot of open questions concerning these Brill-Noether loci. In particular, in general, we do not know when they are non-empty. 

In the first part of the paper, we develop a new theory of walls and chambers on varieties of arbitrary dimension, which is based on the classical theory introduced by Qin in \cite{Qin1}. This new approach allows us, in the second part of the paper, to obtain our main results concerning moduli spaces of stable vector bundles on 3-folds. More precisely, we will consider $\pi:X=\mathbb{P}(\mathcal{E})\rightarrow \mathbb{P}^2$ a ruled 3-fold defined by a rank two vector bundle $\mathcal{E}$ on $ \mathbb{P}^2$. In the literature, there are several works concerning moduli spaces of rank two stable bundles on $ \mathbb{P}^3$ (see for instance \cite{Hartshorne1978}, \cite{Vitter2003} and \cite{Coanda}) but there are scatered results when the underlying variety has reacher geometry in terms, for instance, of its ample cone $\mathcal{C}_X$.

Applying the theory of walls and chambers we will contribute to the better understanding of the geometry of several moduli spaces. Since we are also interested on Brill-Noether problems, we have focused the attention on constructions that give us stable bundles with sections. In \cite{trabajo_futuro} we will apply this theory to describe other moduli spaces. More precisely, the paper is organized as follows.  The first section is devoted to develop a new theory of walls and chambers that will be a key tool to describe, later on, different moduli spaces. Stability becomes more subtle when dealing with vector bundles on higher dimensional varieties, hence the introduction of the theory of walls and chambers helps in this direction and also allows to describe them geometrically. In the second section, we describe moduli spaces of stable vector bundles on ruled 3-folds. We will prove that some of them are empty (see Theorem \ref{th_vacio1}) and, in other cases, we will  describe and compute the dimension of one of its irreducible components (see Theorem \ref{th_irreducible_comp}). Moreover, we will see how these moduli spaces changes as the stability condition varies (see Theorem \ref{prop_decom}).  Finally, we will also  prove that some of them contain non-empty Brill-Noether loci (see Theorem \ref{BN}). 

\begin{notation}
    We will work over $K$ an algebraically closed field of characteristic $0$.   
    Given $X$ an algebraic variety and $\mathcal{F}$ a vector bundle on $X$, if there is no confusion, we will write $\Ho^{i}(\mathcal{F})$ to denote the $i$-th cohomology group $\Ho^{i}(X,\mathcal{F})$ and given a divisor $D$ on $X$, $\Ho^{i}\mathcal{O}_X(D)$ instead of $\Ho^{i}(X,\mathcal{O}_X(D))$.
\end{notation}

\section{Walls and chambers}
Inspired by the theory of walls and chambers introduced by Qin (see \cite{Qin1}),  in this section we develop a slightly different theory that will allow us to describe certain moduli spaces of stable bundles on varieties of dimension $n\geq3$. Later we will apply it in dimension $3$ and, in particular, it will 
anable us study some Brill-Noether loci on higher dimension.

Throught this section, $X$ will be a smooth projective variety of dimension $n$. Recall that a polarization is an element $L\in \Num(X)$ that is ample and the Kähler cone $\mathcal{C}_X$ of $X$ is an open cone in $\Num(X)\otimes \mathbb{R}$ spaned by polarizations.
\begin{definition}
    Let $X$ be a smooth projective variety of dimension $n$. Let $\xi\in\Num(X)\otimes\mathbb{R}$ and $c_i\in\Ho^{2i}(X,\mathbb{Z})$, for $i=1,2$. 
    We define $$W^\xi:=\{x\in\Num(X)\otimes \mathbb{R}|\thinspace \xi\cdot x^{n-1}=0\}\cap \mathcal{C}_X.$$
We define $\mathcal{W}(c_1,c_2)$ to be the set of elements $W^\xi$ where $\xi$ is the numerical equivalence class of a divisor $D$ on $X$ such that
 $\mathcal{O}_X(D+c_1)$ is divisible by $2$ on $\Pic(X)$ and
  there exists a codimension 2 locally complete intersection 
      $Z$ on $X$ such that $$[Z]=c_2+\frac{D^2-{c_1}^2}{4}.$$


A wall of type $(c_1,c_2)$ is an element of $\mathcal{W}(c_1,c_2)$.
A chamber $\mathcal{C}$ is a connected component of $\mathcal{C}_X\backslash\mathcal{W}(c_1,c_2)$ and a $\mathbb{Z}$-chamber is the intersection $\mathcal{C}\cap \Num(X)$. We define a face as $\mathcal{F}:=\Closure(\mathcal{C})\cap W^{\xi}$. 

\end{definition}

\begin{definition}
    Let $\xi\in\Num(X)\otimes\mathbb{R}$ defining a wall $W^{\xi}$ of type $(c_1,c_2)$. We say that $W^{\xi}$ separates two polarizations (respectively two chambers) $L_1$ and $L_2$ (respectively $\mathcal{C}_1$ and $\mathcal{C}_2$),  if $\xi\cdot{L_1}^{n-1}<0< \xi\cdot{L_2}^{n-1}$ (respectively for $L_i\in\mathcal{C}_i$). 
\end{definition}

\begin{remark}\label{remark_finitos_xi}
    If $\xi$ is numerically equivalent to a divisor $D$ and defines a wall $W^{\xi}$ that separates $L_1$ and $L_2$, then, by \cite[Theorem 1.2.5]{Qin1}, there exist integers $i,j$, $0\leq i<j\leq n-1$ such that, if $$S:={L_1}^{n-1-j}\cdot {L_2}^{i}\cdot (L_1+L_2)^{j-i-1},$$
    then $$({c_1}^2-4c_2)\cdot S\leq D^2\cdot S<0.$$
    In particular, if $\xi$ defines a wall $W^{\xi}$ that separates two polarizations, then there exists only a 
 finite number of $\eta\in\Num(X)\otimes \mathbb{R}$  that  defines the wall $W^{\xi}$, this is, such that $W^{\eta}=W^{\xi}$.
\end{remark}

\begin{remark}
    By abuse of notation, we will write $\xi$ instead of the divisor $D\in \Div(X)$ such that $\xi\equiv D$.
\end{remark}

\begin{lemma}
\label{lema_pared}
Let $\xi\in\Num(X)\otimes\mathbb{R}$ defining a non empty wall $W^{\xi}$. Then, there exists a polarization $L\in \Num(X)\otimes\mathbb{R}$ such that $\xi\cdot L^{n-1}>0$.
\end{lemma}
\begin{proof}
Notice that, if $W^{\xi}$ is non empty, then it separates two chambers. Therefore, in particular, there exists an element $L\in \Num(X)\otimes\mathbb{R}$ such that $\xi\cdot L^{n-1}>0$, which corresponds to an element in the chamber above $W^{\xi}$.
\end{proof}
\vspace{3mm}
The main reason to introduce the theory of walls and chambers is to study moduli spaces of $H$-stable bundles and how they vary when we move the ample divisor $H$. Let us briefly recall the definition of stability we will care about and fix the notation we will use.

\begin{definition}
Let $X$ be a smooth projective variety of dimension $n$,  $H$ an ample divisor on $X$ and $E$  a vector bundle on $X$.  Then $E$ is stable with respect to $H$ (or $H$-stable) if, for any subbundle $F$ of $E$, we have $$\mu_H(F):=\frac{c_1(F)\cdot H^{n-1}}{\rank(F)}<\frac{c_1(E)\cdot H^{n-1}}{\rank(E)}:=\mu_H(E)$$
where $c_1(F)$ and $c_1(E)$ stand for the first Chern class of $F$ and $E$. If the inquality is not strict we say that $E$ is $H$-semistable. 
\end{definition}

As in \cite{Qin1}, we can relate our definition of walls and chambers with a notion of equivalence classes of polarizations.
First of all, let us introduce some notions about equivalent classes of polarizations.

\begin{notation}

Let $L_1$ and $L_2$ be two polarizations and let $E$ be a rank 2 vector bundle. We say that $L_1\overset{s}{\geq} L_2$ if and only if $E$ is $L_2$-stable whenever it is $L_1$-stable.
We say that $L_1\overset{s}{=}L_2$ if and only if $L_1\overset{s}{\geq} L_2$ and $L_2\overset{s}{\geq} L_1$ and we set $$\Delta_L:=\{L'| \thickspace L'\overset{s}{\geq} L\} \thickspace\text {and} \thickspace\mathcal{E}_L:=\{L'|\thickspace L'\overset{s}{=} L\}.$$
\end{notation}


\begin{lemma}
\label{lemma1}
    Let $X$ be a smooth projective variety of dimension $n$ and  let $\xi\in\Num(X)\otimes\mathbb{R}$. Assume that there exist two polarizations $L_1$ and $L_2$ on $X$ such that $\xi\cdot {L_1}^{n-1}<0<\xi\cdot {L_2}^{n-1}$. Then, there exists an element $a\in\mathbb{R}^{+}$ such that $\xi\cdot(L_1+aL_2)^{n-1}=0$.
\end{lemma}
\begin{proof}
    Let us consider the following polynomial in $a$ of degree $n-1$ $$P(a):=\xi\cdot(L_1+aL_2)^{n-1}=\sum_{k=0}^{n-1}{{n-1}\choose{k}}({L_1}^{n-1-k}\cdot {L_2}^k\cdot\xi) a^k$$ with coefficients $l_k:={{n-1}\choose{k}}{L_1}^{n-1-k}\cdot {L_2}^k\cdot\xi\in\mathbb{Z}$. In particular, $l_0=\xi\cdot {L_1}^{n-1}<0$ and $l_{n-1}=\xi\cdot {L_2}^{n-1}>0$ and thus $l_0\cdot l_{n-1}<0$, which implies that there exists  a root $a\in \mathbb{R}^{+}$ of $P(a)$.
\end{proof}

\vspace{3mm}

The following result relates chambers and equivalence classes of polarizations.

\begin{lemma}
    \label{lemma_separar} Let $L_1$ and $L_2$ be two polarizations and $E$ a rank $2$ vector bundle with $c_i(E):=c_i$ for $i=1,2$. Assume that $E$ is $L_2$-unstable and $L_1$-stable. Then, there exists $\xi\in\Num(X)\otimes\mathbb{R}$ defining a non-empty wall of type $(c_1,c_2)$ with $\xi\cdot {L_1}^{n-1}<0\leq\xi\cdot {L_2}^{n-1}$ and thus separating $L_1$ and $L_2$.
\end{lemma}

\begin{proof}
    Notice that, since $E$ is $L_2$-unstable,  there exists $\mathcal{O}_X(G)\hookrightarrow E$ a  subbundle of $E$ such that the quotient $E/\mathcal{O}_X(G)$ is torsion free and $G\cdot {L_2}^{n-1}\geq \frac{c_1\cdot {L_2}^{n-1}}{2}$.

    Let us fix $\xi\equiv 2G-c_1$. First of all we will see that $\xi$ defines a wall of type $(c_1,c_2)$.
   By definition, $\mathcal{O}_X(\xi+c_1)$ is divisible by $2$ in $\Pic(X)$.
    Moreover, since $E/\mathcal{O}_X(G)$ is torsion free, we have the short exact sequence $$0\rightarrow \mathcal{O}_X(G)\rightarrow E\rightarrow I_Z(c_1-G)\rightarrow0,$$ where $Z$ is a codimension $2$ locally complete intersection  on $X$ such that $[Z]=\frac{\xi^2-{c_1}^2}{4}+c_2$.

Putting altogether, by definition, $\xi$ defines  a wall $W^{\xi}$ of type $(c_1,c_2)$. Let us see that it is non empty.

 Notice that $G\cdot {L_2}^{n-1}\geq \frac{c_1\cdot {L_2}^{n-1}}{2}$ is equivalent to $\xi\cdot {L_2}^{n-1}\geq0$.
 On the other hand, since $E$ is $L_1$-stable, we have $\xi\cdot {L_1}^{n-1}<0$.
 If $\xi\cdot {L_2}^{n-1}=0$ then $L_2\in W^{\xi}$ and we are done. 
 Assume that $\xi\cdot {L_2}^{n-1}>0$. Then, since $\xi\cdot {L_1}^{n-1}<0<\xi\cdot {L_2}^{n-1}$, by Lemma \ref{lemma1} there exists $a\in \mathbb{R}^+$ such that $\xi\cdot(L_1+aL_2)^{n-1}=0$ and thus $L=L_1+aL_2\in W^{\xi}$. 
\end{proof}

\begin{proposition}
\label{prop_equiv}
    Let $\mathcal{C}$ be a chamber and $L_1,L_2\in\mathcal{C}$ two polarizations. Then $L_1\overset{s}{=}L_2\overset{s}{=}L_1+L_2$.
\end{proposition}

\begin{proof}
    Let us assume that $L_1\not\overset{s}{=}L_2$. Then, by definition, there exists a rank $2$ vector bundle $E$ on $X$ with $c_i(E)=c_i$ such that $E$ is $L_1$-stable but $L_2$-unstable.  By Lemma \ref{lemma_separar}, there exists a non-empty wall separating $L_1$ and $L_2$ what contradicts the fact that $L_1,L_2\in\mathcal{C}$.

 Finally, since by construction a chamber is convex and closed under the action of $\mathbb{R}^+$,   $L_1\overset{s}{=}L_2$ implies that $L_i\overset{s}{=}L_1+L_2$ for $i=1,2$.
\end{proof}

\begin{corollary}
    Every $\mathbb{Z}$-chamber is contained in an equivalence class and every equivalence class is a union of $\mathbb{Z}$-chambers and possible polarizations lying on walls.

\end{corollary}

\begin{proposition}
\label{prop_camarasequiv}
    Let $\mathcal{C}$ and $\mathcal{C'}$ be two chambers having a unique common face that is part of the wall $W$. 
    Let us assume that $\Num(X)\cap \mathcal{C}\neq\emptyset$ and $\Num(X)\cap \mathcal{C'}\neq\emptyset$.
    Then $\Num(X)\cap \mathcal{C}\overset{s}{=}\Num(X)\cap \mathcal{C'}$ if and only if $W$ is not of the form $W^{\xi}$ where $\xi\equiv 2G-c_1$ for some subbundle 
 $\mathcal{O}_X(G)\hookrightarrow E$ of a rank $2$ vector bundle $E$ which is $\mathcal{C}$-stable or $\mathcal{C}'$-stable.
    
\end{proposition}
\begin{proof}
    
  If $\Num(X)\cap \mathcal{C}\overset{s}{\neq}\Num(X)\cap \mathcal{C'}$, then there exists a rank two vector bundle $E$ on $X$ with $c_1(E)=c_1$ and polarizations  $L_1\in\mathcal{C}\cap\Num(X)$ and $L_2\in\mathcal{C'}\cap\Num(X)$ such that $E$ is $L_2$-unstable and $L_1$-stable.
By Lemma \ref{lemma_separar}, there is $\xi\in\Num(X)$ and a non empty wall $W^{\xi}$ which separates $\mathcal{C}$ and $\mathcal{C}'$. Hence, $W=W^{\xi}$, a contradiction.

    Let us now prove the converse.
    Assume that $W$ is of the form $W^{\xi}$ where $\xi\equiv 2G-c_1$ with $\mathcal{O}_X(G)\hookrightarrow E$, being $E$ a rank 2 vector bundle $\mathcal{C}$-stable or $\mathcal{C}'$-stable. Take $L_1\in\mathcal{C}$ and $L_2\in\mathcal{C}'$. Since  we are assuming that $L_1\overset{s}{=}L_2$, $E$ is $L_1$-stable and $L_2$-stable and hence,  $\xi\cdot {L_1}^{n-1}<0$ and $\xi\cdot {L_2}^{n-1}<0$. Thus $W^{\xi}$ can not separate the chambers $\mathcal{C}$ and $\mathcal{C}'$ which is a contradiction.
\end{proof}

\begin{proposition}
    Let $\mathcal{C}_1$ and $\mathcal{C}_2$ be two chambers having a unique common face which is part of the wall $W$. Let $\mathcal{F}=W\cap \Closure(\mathcal{C}_i)\neq\emptyset$, $i=1,2$, be the common face and $L\in\Num(X)\cap\mathcal{F}$. Then,
    \begin{itemize}
        \item[i)] $\Num(X)\cap\mathcal{F}$ is in one equivalence class. 
        \item[ii)] $\Delta_{L}\supseteq (\Num(X)\cap\mathcal{C}_1)\cap(\Num(X)\cap\mathcal{C}_2)$.
        \item[iii)]$\Num(X)\cap\mathcal{C}_1\overset{s}{=}\Num(X)\cap\mathcal{C}_2$ if and only if $ \mathcal{E}_L\supseteq (\Num(X)\cap\mathcal{C}_1)\cup (\Num(X)\cap\mathcal{C}_2)$. 
    \end{itemize}
\end{proposition}
\begin{proof}
    i) Assume that $\Num(X)\cap\mathcal{F}$ is not in an equivalence class. Then, there exists $L_1,L_2\in\Num(X)\cap\mathcal{F}$ and a rank 2 vector bundle $E$ such that $E$ is $L_1$-stable but $L_2$-unstable. By Lemma \ref{lemma_separar}, there exists $\xi\in\Num(X)$ defining a non empty wall $W^{\xi}$ 
    separating $L_1$ and $L_2$ and hence $W^{\xi}\cap \mathcal{F}\neq\emptyset$, which is a contradiction.

    ii) It is enough to prove that $\Delta_L\supseteq \Num(X)\cap\mathcal{C}_1$. Let $L_1\in \Num(X)\cap\mathcal{C}_1$ and assume that $L_1\not\in \Delta_L$. Then there exists  a rank 2 vector bundle $E$ $L$-stable but $L_1$-unstable. By Lemma \ref{lemma_separar}, there exists a non empty wall $W^{\xi}$ separating $L_1$ and $L$. Moreover, since $\xi\cdot L^{n-1}<0$, we get $L\not\in W^{\xi}$. But the only wall separating $L$ and $L_1$ is $W$, which contains $L$, and hence we get a contradiction. 

    iii) By definition, if $\mathcal{E}_L\supseteq (\Num(X)\cap\mathcal{C}_1)\cup (\Num(X)\cap\mathcal{C}_2)$ then $(\Num(X)\cap\mathcal{C}_1)\overset{s}{=} (\Num(X)\cap\mathcal{C}_2)$. Let us now prove the converse. 
    Assume that $(\Num(X)\cap\mathcal{C}_1)\overset{s}{=} (\Num(X)\cap\mathcal{C}_2)$ and consider a rank 2 vector bundle $E$ which is $\mathcal{C}_i$-stable for $i=1$ or $i=2$. If $E$ is not $L$-stable, then by Lemma \ref{lemma_separar} there exists a wall $W^{\xi}$ with $\xi=2G-c_1$ for some $\mathcal{O}_X(G)\hookrightarrow E$, that separates $L$ and $L_i$. 
     But since $L$ belongs to $W\cap \Closure(\mathcal{C})$ and there is a unique common face in the wall $W$, $W=W^{\xi}$.
    But if $W=W^{\xi}$, then $\xi$ separates $L_1\in\mathcal{C}_1$ and $L_2\in\mathcal{C}_2$ and $\xi \cdot L_1^{n-1}>0>\xi\cdot L_2^{n-1}$ contradicts the fact that $(\Num(X)\cap\mathcal{C}_1)\overset{s}{=} (\Num(X)\cap\mathcal{C}_2)$. Thus $E$ must be $L$-stable and therefore $L\overset{s}{=}\Num(X)\cap \mathcal{C}_i$ and $\mathcal{E}_L\supseteq (\Num(X)\cap\mathcal{C}_1)\cup (\Num(X)\cap\mathcal{C}_2)$.
\end{proof}

\begin{remark}
\label{M_cara}
From the above proposition, $E$ is $(\Num(X)\cap\mathcal{F})$-stable if and only if $E$ is $(\Num(X)\cap\mathcal{C}_i)$-stable for $i=1,2$.
\end{remark}

Now we want to construct families $E_\xi(c_1,c_2)$ of rank 2 vector bundles on $X$ as in \cite{Qin1}, but in the context of $X$ being an algebraic variety of dimension greater or equal than two, not necessarily a surface. 
In this case the situation is more tricky and we will use the well known Hartshorne-Serre correspondence which we are going to recall briefly. 

Assume that $Z$ is the dependency locus of $r-1$ sections $\alpha_i$ of a rank $r$ vector bundle $E$ over $X$ with $\bigwedge^{r}E=\mathcal{L}$. This produces  the exact sequence $$0\rightarrow\mathcal{O}_X^{\oplus r-1}\overset{\alpha_1,\dots,\alpha_{r-1}}{\longrightarrow} E\rightarrow I_{Z}\otimes\mathcal{L}\rightarrow0,$$ which implies that $\bigwedge^{2}(\mathcal{N}_{Z/X})\otimes\mathcal{L}^*_{|Z}$ is globally generated by $r-1$ sections, where  $\mathcal{N}_{Z/X}$ is the normal bundle to $Z$ in $X$. Hartshorne-Serre correspondence consists of reversing this process. The statement is the following:

\begin{theorem}
    Let \( X \) be a smooth algebraic variety and let \( Z \) be a local complete intersection subscheme of codimension two in \( X \). Let \(\mathcal{N}_{Z/X} \) be the normal bundle of \( Z \) in \( X \) and let \( \mathcal{L}  \) be a line bundle on \( X \) such that \( H^2(X, \mathcal{L} ^*) = 0 \). Assume that \( \bigwedge^{2}\mathcal{N}_{Z/X} \otimes \mathcal{L} ^*|_Z \) has \( r - 1 \) generating global sections \( s_1, \dots, s_{r-1} \). Then, there exists a rank \( r \) vector bundle \( E \) over \( X \) given by a non trivial extension $$0\rightarrow\mathcal{O}_X^{\oplus r-1}\rightarrow E\rightarrow I_Z\otimes\mathcal{L}\rightarrow0$$  such that:

\begin{enumerate}
    \item[(i)] \( \bigwedge^r E = \mathcal{L} \);
    \item[(ii)] \( E \) has \( r - 1 \) global sections \( \alpha_1, \dots, \alpha_{r-1} \) whose dependency locus is \( Z \) and such that
    \[
    s_1 \alpha_1|_Y + \dots + s_{r-1} \alpha_{r-1}|_Z = 0.
    \]
\end{enumerate}


\end{theorem}
\begin{proof}
    See \cite{Arrondo}, \cite{Grauert} and \cite{Vogelaar}.
\end{proof}

In particular, we can construct a rank $2$ vector bundle $E$ by an non trivial extension of type $$0\rightarrow \mathcal{O}_X\rightarrow E\rightarrow I_Z(D)\rightarrow0,$$ if $Z$ is a codimension $2$ locally complete intersection,   $\Ho^2(X,\mathcal{O}_X(-D))=0$ and $\Ho^0(Z,{\bigwedge^2\mathcal{N}}_{Z/ X}\otimes {\mathcal{O}_X(-D)}_{|Z})\neq0$, where $\mathcal{N}_{Z/ X}$ is the normal bundle of $Z$ in $X$.

\begin{definition}
    Let us consider $\xi\in \Num(X)\otimes\mathbb{R}$ defining a non empty wall $W^{\xi}$ of type $(c_1,c_2)$ and $[Z]=\frac{\xi^2-{c_1}^2}{4}+c_2$. We say that $\xi$ satisfies the $C-$condition if 
$\Ho^2(X,\mathcal{O}_X(\xi))=0$ and $\Ho^0(Z,{\bigwedge^2\mathcal{N}}_{Z/ X}\otimes\mathcal{O}_X(\xi)_{|Z})\neq0$.

\end{definition}

\begin{definition}
\label{def_Exi}
    Let $X$ be an algebraic variety of dimension $n$ and $\xi\in\Num(X)\otimes\mathbb{R}$ defining a non empty wall $W^{\xi}$ of type $(c_1,c_2)$  satisfying the $C-$condition.
We define $E_{\xi}(c_1,c_2)$ to be the set of rank two vector bundles $E$ on $X$ given by a non trivial extension of type 
\begin{equation}
\label{E_chi}
    0\rightarrow \mathcal{O}_X(G)\rightarrow E\rightarrow I_Z(c_1-G)\rightarrow0,
\end{equation}
 where $G$ is a divisor on $X$ such that $\xi\equiv 2G-c_1$ and $Z$ is a codimension $2$ locally complete intersection  such that $[Z]=c_2+\frac{\xi^2-{c_1}^2}{4}$.
\end{definition}

Notice that, by  Hartshorne-Serre correspondence, the fact that $\xi$ satisfies the $C-$condition allow us to guarantee that $E$ is a vector bundle and from (\ref{E_chi}) we have  $c_1(E)=c_1$ and $c_2(E)=c_2$.

Moreover, the families $E_{\xi}(c_1,c_2)$ allow us to construct stable vector bundles. In fact, following the proof of \cite[Theorem 1.2.3]{Qin1}, we get:

\begin{theorem}
\label{th_Exi}
    Let $X$ be an algebraic variety of dimension $n$ and $\xi\in\Num(X)\otimes\mathbb{R}$ defining a non empty wall $W^{\xi}$ of type $(c_1,c_2)$ satisfying the $C-$condition. Let us take $E\in E_{\xi}(c_1,c_2)$. The following holds: 
    \begin{itemize}


       \item[i)] If $L_1$ is a polarization such that $\xi\cdot {L_1}^{n-1}>0$, then $E$ is $L_1$-unstable.

     \item[ii)]  If $\mathcal{C}$ is a chamber such that $W^{\xi}\cap \Closure(\mathcal{C})\neq\emptyset$ and $\xi\cdot L^{n-1}<0$ for $L\in\mathcal{C}$, then $E$ is $L$-stable for any polarization $L\in\mathcal{C}$.
 \item[iii)]  If $\mathcal{C}$ is a chamber such that $W^{\xi}\cap \Closure(\mathcal{C})\neq\emptyset$ and $\xi\cdot L^{n-1}\leq 0$ for $L\in \mathcal{C}$, then $E$ is $L$-semistable for any polarization $L\in\mathcal{C}$ and $E$ is strictly $L$-semistable for $L\in W^{\xi}\cap\Num(X)$.
    \end{itemize}

\end{theorem}

\begin{proof}

    i) Notice that $\xi\cdot {L_1}^{n-1}>0$ is equivalent to $G\cdot {L_1}^{n-1}>\frac{c_1\cdot {L_1}^{n-1}}{2}$. Hence the line bundle $\mathcal{O}_X(G)\hookrightarrow E$   $L_1$-desestabilizes $E$. 

    ii)  By definition, any $E\in E_{\xi}(c_1,c_2)$ is given by a non trivial extension of type 
    \begin{equation}
        \label{ext}0\rightarrow \mathcal{O}_X(G)\rightarrow E\rightarrow I_Z(c_1-G)\rightarrow0,
    \end{equation}
    where $G$ is a divisor on $X$ such that $\xi\equiv 2G-c_1$ and $Z$ is a locally complete intersection codimension 2 cycle such that $[Z]=c_2+\frac{\xi^2-{c_1}^2}{4}$.

Let us consider $\mathcal{O}_X(D)\hookrightarrow E$ a subbundle of $E$ such that the quotient $E/\mathcal{O}_X(D)$ is torsion free.
Since $E$ is given by (\ref{ext}), we have two possibilities: either $\mathcal{O}_X(D)\hookrightarrow \mathcal{O}_X(G)$ or $\mathcal{O}_X(D)\hookrightarrow I_Z(c_1-G)$.

Assume that $\mathcal{O}_X(D)\hookrightarrow \mathcal{O}_X(G)$. In this case, $G-D$ is effective, which implies that $(G-D)\cdot L^{n-1}\geq 0$. Hence, $$D\cdot L^{n-1}\leq G\cdot L^{n-1}< \frac{c_1\cdot L^{n-1}}{2},$$ where the last inequality follows from the fact that $\xi\cdot L^{n-1}<0$.

Let us now assume that $\mathcal{O}_X(D)\hookrightarrow I_Z(c_1-G)$.
Notice that $c_1-G-D$ is an effective divisor. If $c_1-G-D=0$, 
we have $\h^0I_Z>0$ which implies that $[Z]=0$ and in this case we would have both a subbundle and quotient bundle of $E$, which is a contradiction. Hence $c_1-G-D$ is strictly effective and thus $(c_1-G-D)\cdot {L'}^{n-1}>0$ for any  divisor $L'\in (\Num(X)\otimes \mathbb{R})\cap \mathcal{C}_X$. In particular, $(c_1-G-D)\cdot {L'_0}^{n-1}>0$ for $L'_0\in W^{\xi}\cap \Closure(\mathcal{C})$. Moreover, by definition,  $\xi\cdot (L'_0)^{n-1}=0$. Putting altogether, 
\begin{equation}
    \label{et1}(2D-c_1)\cdot (L'_0)^{n-1}<0.
\end{equation}

We want to prove that $D\cdot L^{n-1}<\frac{c_1\cdot L^{n-1}}{2}$. Assume that $D\cdot L^{n-1}\geq\frac{c_1\cdot L^{n-1}}{2}$, which is equivalent to
\begin{equation}
\label{et2}
    (2D-c_1)\cdot L^{n-1}\geq0.
\end{equation}

Let us define $\xi'\equiv 2D-c_1$. Notice that $\xi'$ defines a wall $W^{\xi'}$ of type $(c_1,c_2)$.
In fact, by definition $\xi'$ is divisible by 2 in $\Pic(X)$.
Since $\mathcal{O}_X(D)\hookrightarrow E$ and $E/\mathcal{O}_X(D)$ is torsion free, $E$ sits on the short exact sequence $$0\rightarrow \mathcal{O}_X(D)\rightarrow E\rightarrow I_{Z'}(c_1-D)\rightarrow0,$$
where $Z'$ is a locally complete intersection of codimension 2 such that $[Z']=c_2+\frac{(\xi')^2-{c_1}^2}{4}$. 

Therefore, we see that $\xi'$ defines a  wall $W^{\xi'}$ of type $(c_1,c_2)$.
Finally, we will see that it is non empty.
By (\ref{et1}) and (\ref{et2}), we have  $\xi'\cdot (L'_0)^{n-1}<0\leq \xi'\cdot L^{n-1} $.
If $\xi'\cdot L^{n-1}=0$, by definition $L\in W^{\xi'}$. Otherwise,
$\xi'\cdot (L'_0)^{n-1}<0< \xi'\cdot L^{n-1} $ implies by Lemma \ref{lemma1} that $\xi\cdot(L'_0+aL)^{n-1}=0$ for some $a\in\mathbb{R}^+$ and thus $L'_0+aL\in W^{\xi'}$.
Hence, there exists a non empty wall $W^{\xi'}$ which separates $L'_0$ and $L$ and clearly $L'_0\not\in W^{\xi'}$. But, by construction, for any $L\in \mathcal{C}$ and $L'_0\in W^{\xi}\cap \Closure(\mathcal{C})$, any wall separating $L$ and $L'_0$ must contain $L'_0$. Therefore we get a contradiction and hence $D\cdot L^{n-1}<\frac{c_1\cdot L^{n-1}}{2}$.

Putting altogether, $\mu_L(\mathcal{O}_X(D))<\mu_L(E)$ for any $\mathcal{O}_X(D)\hookrightarrow E$, which proves the $L$-stability of $E$.

    ii) Using the same arguments as in i) we prove that $E$ is $L$-semistable for $L\in\mathcal{C}$. On the other hand, if $L_0\in W^{\xi}$, then $\xi\cdot {L_0}^{n-1}=0$, which is equivalent to $G\cdot {L_0}^{n-1}=\frac{c_1\cdot {L_0}^{n-1}}{2}$ and therefore $E$ is strictly $L_0$-semistable.
\end{proof}


\begin{notation}
    Given a chamber $\mathcal{C}$, let us denote by $M_{\mathcal{C}}(2;c_1,c_2)$ the moduli space of rank $2$, $L$-stable vector bundles $E$ on $X$ with $c_1(E)=c_1$ and $c_2(E)=c_2$ for any polarization $L\in \mathcal{C}$.
Equivalently, given a face $\mathcal{F}$, let us denote by $M_{\mathcal{F}}(2;c_1,c_2)$ the moduli space of rank $2$, $L$-stable vector bundles $E$ on $X$ with $c_1(E)=c_1$ and $c_2(E)=c_2$ for any polarization $L\in \mathcal{F}$.
\end{notation}

As a consequence of Theorem \ref{th_Exi} we get the following result.

\begin{corollary}
\label{cor_inclusion}
    Let $\xi\in\Num(X)\otimes\mathbb{R}$ defining a non empty wall $W^{\xi}$ and satisfying the $C-$condition. Assume that $W^{\xi}\cap \Closure(\mathcal{C})\neq\emptyset$ and $\xi\cdot L^{n-1}<0$ for  $L\in\mathcal{C}$.  Then, there exists an embedding $E_{\xi}(c_1,c_2)\hookrightarrow M_{\mathcal{C}}(2;c_1,c_2)$. 
\end{corollary}

For the proof we need the following technical Lemma.

\begin{lemma}
\label{lemma_aut}
    Let $\xi\in \Num(X)\otimes\mathbb{R}$  defining a non empty wall $W^{\xi}$ and satisfying the $C-$condition, $E\in E_{\xi}(c_1,c_2)$ and $\mathcal{O}_X(G)\hookrightarrow E$ its subbundle coming from the construction of $E$ in Definition \ref{def_Exi}. Assume that there exists $L'\in\Num(X)$ such that $\xi\cdot (L')^{n-1}<0$. Then, $\Hom(\mathcal{O}_X(G), E)\cong K$. Moreover, if $\mathcal{O}_X(D)\hookrightarrow E$ with $E/\mathcal{O}_X(D)$ torsion free and $\xi\equiv 2D-c_1$, then $$\mathcal{O}_X(D)\cong \mathcal{O}_X(G).$$
\end{lemma}

\begin{proof}
    First of all let us prove that $\Hom(\mathcal{O}_X(G), E)\cong \Ho^0E(-G)\cong K$.
    Twisting the exact sequence defining $E\in E_{\xi}(c_1,c_2)$ by $\mathcal{O}_X(-G)$ and taking cohomology, we get the long exact sequence
    $$0\rightarrow \Ho^0\mathcal{O}_X\rightarrow \Ho^0E(-G)\rightarrow \Ho^0I_Z(c_1-2G) \rightarrow\cdots$$

\noindent    \textbf{Claim: $\Ho^0I_Z(c_1-2G)=0$.}
    
  \noindent  \textbf{Proof of the Claim:}
    Since $c_1-2G\equiv -\xi$ and by Lemma \ref{lema_pared} there exists a polarization $L$ such that $\xi\cdot L^{n-1}>0$, it follows that $c_1-2G$ is non effective and thus  $\Ho^0I_Z(c_1-2G)=0$.
    Hence $\Ho^0E(-G)\cong \Ho^0\mathcal{O}_X\cong K$.

    Let us prove now the second statement. Assume that $\xi\equiv 2D-c_1$.
    Notice that $\xi\equiv 2D-c_1$ is equivalent to $c_1\equiv 2D-\xi\equiv 2(D-G)+c_1$ and $0\equiv 2(D-G)$, which implies $D\equiv G$. Hence $c_1-G-D\equiv c_1-2G\equiv-\xi$.
On the other hand, there exists a polarization $L$ such that $\xi\cdot L^{n-1}>0$ and thus, since $-\xi\equiv c_1-G-D\equiv c_1-2G$,  $c_1-G-D$ and $ c_1-2G$ are not effective divisors. 
In particular, $$\Hom(\mathcal{O}_X(G), E/\mathcal{O}_X(D))=\Hom(\mathcal{O}_X,I_Z(c_1-D-G))=0$$
and $$\Hom(\mathcal{O}_X(D), E/\mathcal{O}_X(G))=\Hom(\mathcal{O}_X,I_Z(c_1-G-D))=0.$$
Hence, the conclusion follows from the standart fact that asserts that two invertible subsheaves of a sheaf with torsion free quotients coincide if the map of one to the quotient by the other is zero.
\end{proof}

Now we can prove Corollary \ref{cor_inclusion}.

\begin{proof}
     By Theorem \ref{th_Exi}, there exists a map $\phi: E_{\xi}(c_1,c_2)\rightarrow M_\mathcal{C}(2;c_1,c_2)$. Assume that it is not an embedding and $E\in E_{\xi}(c_1,c_2)$ is given by two different extensions, which by Lemma \ref{lemma_aut} we can assume that are:

     $$e_1: 0\rightarrow \mathcal{O}_X(G)\overset{\alpha_1}{\rightarrow} E \overset{\alpha_2}{\rightarrow} I_{Z_1}(c_1-G)\rightarrow0$$
     and 
     $$e_2: 0\rightarrow \mathcal{O}_X(G)\overset{\beta_1}{\rightarrow} E \overset{\beta_2}{\rightarrow} I_{Z_2}(c_1-G)\rightarrow0.$$

     Since $\Hom(\mathcal{O}_X(G),I_{Z_1}(c_1-G))\cong \Ho^0 I_{Z_1}(c_1-2G)=0$ and $\Hom(\mathcal{O}_X(G),I_{Z_2}(c_1-G))\cong \Ho^0 I_{Z_2}(c_1-G)=0$, we get
     $
         \beta_2\circ \alpha_1= 0= \alpha_2\circ\beta_1.
     $
    Therefore, there exists $\gamma\in Aut(\mathcal{O}_X(G))\cong K$ such that $\alpha_1=\gamma\circ \beta_1$. Hence, $\phi$ is an injection.
\end{proof}

Therefore, we conclude that we can construct $\mathcal{C}$-stable vector bundles by elements in $E_{\xi}(c_1,c_2)$ and hence we can generalize the construction given in \cite{Qin1} for the case of higher dimensional varieties.

In the case of surfaces, if $\mathcal{C}$ and $\mathcal{C}'$ are two chambers with a unique common wall $W^\xi$, one can say more and describe the moduli space $M_{\mathcal{C}}(c_1,c_2)$ in terms of the moduli $M_{\mathcal{C'}}(c_1,c_2)$ and the families $E_{\xi}(c_1,c_2)$. This is, when the underlying variety is a surface we have the following set theoretical decomposition $$M_{\mathcal{C}}(c_1,c_2)=M_{\mathcal{C}'}(c_1,c_2)\sqcup (\bigsqcup_{\xi}E_{\xi}(c_1,c_2)),$$ where $\xi$ runs over all numerical classes defining the wall $W^{\xi}$
 (see \cite[Proposition 1.3.1]{Qin1}).

 While dealing with higher dimensional varieties, this  is more tricky, since it is not true  that for every $E\in M_{\mathcal{C}}(c_1,c_2)\backslash M_{\mathcal{C}'}(c_1,c_2)$ we get an element in $E_{\xi}(c_1,c_2)$. Any $E\in M_{\mathcal{C}}(2;c_1,c_2)\backslash M_{\mathcal{C}'}(2;c_1,c_2)$ is given by an extension of type $$0\rightarrow \mathcal{O}_X(G)\rightarrow E\rightarrow I_Z(c_1-G)\rightarrow0,$$
 but if $\xi\equiv 2G-c_1$, then $\xi$ does not always  satisfy $\Ho^2(Z,\mathcal{O}_X(\xi)_{|Z})=0$.

Fortunately, in next section we will show the existence of some moduli spaces $M_{\mathcal{C}}(2;c_1,c_2)$ which are completely described in terms of the families $E_{\xi}(c_1,c_2)$.

\section{Moduli spaces on ruled 3-folds}

In this section, we will analize a case in which  the moduli space $M_{\mathcal{C}}(2;c_1,c_2)$ is completely decribed by  families of type $E_{\xi}(c_1,c_2)$. As a by product we will be able to obtain results concerning non emptiness of Brill-Noether loci. As we have said in the previous section, in general, one can not expect such decomposition when dealing with rank $2$ vector bundles on higher dimensional varieties. As we pointed out in the introduction, we will focus our attention on moduli spaces of stable rank 2 vector bundles on ruled 3-folds. Let us start introducing them and fixing some notation. 

\vspace{3mm}
Let $\mathcal{E}$ be a rank $2$ vector bundle over $\mathbb{P}^2$ with $c'_i:=c_i(\mathcal{E})$. We will assume that $\mathcal{E}$ is generated by global sections. Hence, $c'_1\geq0$ and ${c'_1}^2\geq c'_2\geq0$. Let us consider $\pi:X=\mathbb{P}(\mathcal{E})\rightarrow \mathbb{P}^2$ the projectivized  vector bundle associated to $\mathcal{E}$ with the natural projection $\pi$, that is, $X$ is a ruled 3-fold over $\mathbb{P}^2$ which is a smooth projective variety of dimension $3$. The tautological line bundle $S$ on $X$  is uniquely determined by the
conditions $S_{|F}\cong \mathcal{O}_F(1)$ and  $\pi_{*}S\cong \mathcal{E}$, being $F$ a fiber of $\pi$. 

Since $\Pic(X)\cong \mathbb{Z}S\bigoplus\pi^*\Pic(\mathbb{P}^2)$,  
if we denote by $H=\pi^*h$, being $h$  the hyperplane class of $\mathbb{P}^2$, any line bundle $L$ on $X$ can be identified with $\mathcal{O}_X(aS+bH)$ with $a,b\in\mathbb{Z}$. It is well known that $\Num(X)\cong \mathbb{Z}\bigoplus \mathbb{Z}$ is generated  by $S$ and $H$
with the intersection product given by $$H^{3}=0, S\cdot H^2=1, S^2\cdot H=c'_1$$ and $$0=\sum_{i=0}^2(-1)^{i}c'_iS^{2-i}\cdot H^{i}.$$

In addition, from this last relation we deduce that $S^3={c'_1}^2-c_2'$.  
Notice that $D=aS+bH$ is numerically equivalent to $0$ if and only if $D=0$. Hence $\Pic(X)\cong \Num(X)$.

It is well known that we can relate the cohomology of line bundles on $X$ with the one of line bundles on $\mathbb{P}^2$. More precisely, if $a\geq0$, then
\begin{equation}
    \label{coh_ruled}
    \Ho^{i}(X,\mathcal{O}_X(aS+b H))=\Ho^{i}(\mathbb{P}^2,S^{a}(\mathcal{E})\otimes \mathcal{O}_{\mathbb{P}^2}(b)),
\end{equation}

\noindent where $S^{a}(\mathcal{E})$ is the $a$-th symmetric power of $\mathcal{E}$.
Moreover, for $-2<a<0$, $$\Ho^{i}(X,\mathcal{O}_X(a S+bH))=0$$ and the case $a\leq -2$ follows from Serre duality. Finally, we recall that the canonical divisor is given by $K_X=-2S+(\pi^*c'_1+\pi^*K_{\mathbb{P}^2}))=-2S+(c'_1-3)H$.

\begin{lemma}
  An element $L=\alpha S+\beta H\in\Pic(X)$ with $\alpha\geq1$ and $\beta\geq2$, is an ample divisor on $X$. 
\end{lemma}
\begin{proof}
    Since $\mathcal{E}$ is globally generated, $\alpha S+\pi^*D$ is an ample divisor on $X$ if $\alpha\geq1$ and $D$ is an ample divisor on $\mathbb{P}^2$, this is, $D=\beta h$ with $\beta >0$. Hence $D= \alpha S+\beta H$ with $\alpha \geq1$ and $\beta\geq2$ is ample.
\end{proof}

\begin{definition}
Assume that $\xi= a S+bH\in \Num(X)$ with $a<0$ is a class defining a wall $W^\xi$.
We say that $L$ is above (below) the wall $W^\xi$ if $\xi\cdot L^{2}>0$ ($\xi\cdot L^{2}<0$).

We say that a wall $W$ is under (above) the wall $W^{\xi}$ if, for $L\in W$, $\xi\cdot L^{2}>0$ ($\xi\cdot L^{2}<0$).
\end{definition}
    
\begin{remark}
    Observe that any $\xi$ defining a wall has integer coefficients.
\end{remark}

The goal of this section is to  describe the moduli space $M_{L}(2; S, (b+1)SH-b^2H^2)$ for some integer $b$ and $L$ moving in a special chamber $\mathcal{C}_b$. To achieve it, we will use the powerfull theory of walls and chambers developed in the previous section.
Let us start with the following two Lemmas.

\begin{lemma}
\label{misma_pared}
    If $\xi$ and $\eta$ are two elements in $\Num(X)\otimes\mathbb{R}$ defining the same non empty wall of type $(c_1,c_2)$, then $\eta = t\xi$ for some non zero $t\in \mathbb{Z}$.
\end{lemma}
\begin{proof}
   Let $\xi= a S+bH \in\Num(X)\otimes\mathbb{R}$ defining the non empty wall $W^{\xi}$ and consider $\eta\in \Num(X)\otimes\mathbb{R}$ defining a non empty wall $W^{\eta}$ such that $W^{\eta}=W^{\xi}$. Notice that we can always find $D=c S+dH\in\Num(X)\otimes\mathbb{R}$ such that $\eta=\xi +D$. Let us see that $D=l\xi$ for some $l\in\mathbb{Z}$.
   
Consider
   $L= \alpha S+\beta H\in W^\xi$.
   Then, 
   $$\xi\cdot L^2 =(aS+bH)\cdot (\alpha S+\beta H)^2 
      =a(\alpha^2S^3+2\alpha\beta S^2\cdot H+\beta^2S\cdot H^2)+b(\alpha^2S^2\cdot H+2\alpha\beta S\cdot H^2) =0, $$
which is equivalent to 
\begin{equation}
    \label{L_pared}
    \frac{\alpha(\alpha c_1'+2\beta)}{\alpha^2[(c_1')^2-c_2']+2\alpha\beta c_1'+\beta^2}=-\frac{a}{b}.
\end{equation}

On the other hand, since $W^{\xi}=W^{\eta}$, we have $$0=\eta\cdot L^2=(\xi +D)\cdot L^2=\xi\cdot L^2+D\cdot L^2= D\cdot L^2$$
and $D\cdot L^2=0$ is equivalent to  $$c=-d\frac{\alpha(\alpha c_1'+2\beta)}{\alpha^2[(c_1')^2-c_2']+2\alpha\beta c_1'+\beta^2}.$$ Hence, by (\ref{L_pared}) we get $\frac{a}{b}=\frac{c}{d}$ and $$D=cS+dH=\frac{d a}{b}S+dH=\frac{d}{b}(aS+bH)=\frac{d}{b}\xi.$$  Notice that, since $D\in \Pic(X)$, $\frac{d}{b}\in\mathbb{Z}$ and this proves the Claim.
Finally, taking $t=1+\frac{d}{b}$ we get $\eta=t\xi$.
\end{proof}


\begin{lemma}
\label{lemma_chib}
    Let $X=\mathbb{P}(\mathcal{E})\rightarrow\mathbb{P}^2$ be a ruled $3$-fold  over $\mathbb{P}^2$ and  consider $\xi_b= -S+2bH$ in $\Num(X)$ with $c_1'+1\leq b\in\mathbb{Z}$.
    Then:
    \begin{itemize}
        \item[i)] $\xi_b$ defines a non empty wall $W^{\xi_b}$ of type $(S, (b+1)SH-b^2H^2)$ and satisfies the $C-$condition.
        \item[ii)] The wall $W^{\xi_b}$ is only defined by $ \xi_b$ and $-\xi_b$, but $-\xi_b$ does not satisfy the $C-$condition. 
    \end{itemize}

\end{lemma}
\begin{proof}

i) Let us first prove that $\xi_b$ defines a  wall $W^{\xi_b}$ of type $(S, (b+1)SH-b^2H^2)$. In fact,
 $\xi_b+S= 2bH$ is clearly divisible by 2 in $\Pic(X)$ and
  $$[Z]=\frac{{\xi_b}^2-S^2}{4}+(b+1)SH-b^2H^2=SH,$$  is a locally complete intersection in $X$.
Moreover,  this wall is non empty. If $L= \alpha S+\beta H$,  then $\xi_b\cdot L^2=0$ if and only if

$\begin{array}{lll}
    \beta &=  \alpha\cdot\left(\frac{2(2b-c'_1)\pm\sqrt{2^2(2b-c_1')^2-4({c'_1}^2-c'_2-2bc_1')}}{-2}\right)
     & =\alpha\cdot \left( 2b-c'_1\pm\sqrt{2b(2b-c_1')+c_2'}\right).
\end{array}$
 
In particular, since $2b>c_1'$, $$\beta=\alpha\cdot \left( 2b-c'_1+\sqrt{2b(2b-c_1')+c_2'}\right)\in \mathbb{R}^+,$$ and thus, if $\alpha\geq1$, $L= \alpha S+ \beta H\in (\Num(X)\otimes \mathbb{R})\cap \mathcal{C}_X,$ which implies  $L\in W^{\xi_b}$.


Let us now see that $\xi_b$ satisfies the $C-$condition.
First of all, let us see that
$$\Ho^0(Z,{{\bigwedge}^2\mathcal{N}}_{Z/ X}\otimes \mathcal{O}_X(\xi_b)_{|Z})\neq0.$$ Since $[Z]=SH$ is a  complete intersection, we have the following exact sequence 
\begin{equation}
    \label{suc_CI}
    0\rightarrow  \mathcal{O}_X(-S-H)\rightarrow\mathcal{O}_X(-S)\oplus \mathcal{O}_X(-H)\rightarrow I_Z\rightarrow0.
\end{equation}

The dual of the normal bundle, $({\mathcal{N}}_{Z/ X})^\vee$, is isomorphic to $I_Z/(I_Z)^2\cong I_Z\otimes \mathcal{O}_Z$. Twisting the exact sequence (\ref{suc_CI}) by $\mathcal{O}_Z$, we get  $I_Z\otimes \mathcal{O}_Z\cong \mathcal{O}_Z(-S)\oplus\mathcal{O}_Z(-H)$ and thus
$${\bigwedge}^2{\mathcal{N}}_{Z/ X}\cong \mathcal{O}_X(S+H)_{|Z}= \mathcal{O}_Z(S^2H+SH^2).$$

On the other hand, $${\mathcal{O}_X(\xi_b)}_{|Z}\cong \mathcal{O}_Z((-S+2bH)\cdot SH)=\mathcal{O}_Z(-S^2H+2bSH^2).$$

Puting altogether, since $b>1$, we have $$\Ho^0(Z,{\bigwedge}^2{\mathcal{N}}_{Z/ X}\otimes \mathcal{O}_X(\xi_b)_{|Z})=H^0(Z,\mathcal{O}_Z((2b+1)SH^2)\neq0.$$ 

\noindent Finally, it follows from (\ref{coh_ruled})  that  $\Ho^2\mathcal{O}_X(\xi_b)=\Ho^2\mathcal{O}_X(-S+2bH)=0$.

Hence $\xi_b$ defines a non empty wall $W^{\xi_b}$ of type $(S, (b+1)SH-b^2H^2)$ and satisfies the $C-$condition.

ii) By Lemma \ref{misma_pared}, the wall $W^{\xi_b}$ is defined by $t\xi_b=-tS+2btH\in\Num(X)$ with $t\neq0$. If $\eta=t\xi_b$ defines a non emty wall of type $(S,(b+1)SH-b^2H^2)$, then $$[Z]=\frac{t^2(\xi_b)^2-S^2}{4}+(b+1)SH-b^2H^2 $$ is a codimension 2 locally complete intersection. Then $[Z]\cdot L\geq0$ for any $L\in\mathcal{C}_X$. Let us consider $L_0=2 S+\beta H$ with   $\beta\gg0$. Then $[Z]\cdot L_0\geq0$ which implies $$\frac{t^2-1}{4}c_1'+(b+1-t^2b)=(t^2-1)[\frac{(c_1')^2}{4}-b]+1\geq0.$$  Thus, since $\frac{c_1'}{4}-b\leq-1$ and $0\neq t$, we must have $|t|=1$.

Moreover, since $-\xi_b+S=2(S-bH)$, $(-\xi_b)^2=\xi_b^2$ and $-\xi_b\cdot L^2=0$, we can conclude from i) that $-\xi_b$ defines a non empty wall of type $(c_1,c_2)$.

Finally, since $b\geq c_1'+1$, we get $$\Ho^0 \left({\bigwedge}^2{\mathcal{N}}_{Z/ X}\otimes\mathcal{O}_X(-\xi_b)_{|Z}\right)\cong \Ho^0\left(\mathcal{O}_Z(c_1'+1-2b)\right)=0$$ and hence $-\xi_b$ does not satisfy the $C-$condition.
\end{proof}
\vspace{0.3cm}

\begin{notation} For any $l\in\mathbb{Z}$, denote by $\xi_l:=-S+2lH$. Given $b\geq c_1'+1$ an integer, let   $W^{\xi_b}$ be the wall defined by $\xi_b$ and  denote by $\mathcal{C}_b$ the chamber such that $W^{\xi_b}\cap \Closure(\mathcal{C}_b)\neq\emptyset$ and $\xi_b\cdot L^{2}<0$ for $L\in \mathcal{C}_b$ and  by $\mathcal{C}'_b$ the chamber such that $W^{\xi_b}\cap \Closure(\mathcal{C}'_b)\neq\emptyset$ and $\xi_b\cdot L^{2}>0$ for $L\in \mathcal{C}'_b$ . 
\end{notation}

Let us start by studying the non emptiness of  $M_{\mathcal{C}_b}(2;S,(b+1)SH-b^2H^2)$.

 \begin{theorem}\label{th_irreducible_comp}Let $X=\mathbb{P}(\mathcal{E})\rightarrow\mathbb{P}^2$ be a ruled  $3$-fold  over $\mathbb{P}^2$ with $c'_i=c_i(\mathcal{E})$.  Let us fix $c_1'+1\leq b\in\mathbb{Z}$ and Chern classes $c_1=S$ and $c_2=(b+1)SH-b^2H^2$. 
 Then,  $$M_{\mathcal{C}_b}(2;c_1,c_2)\neq\emptyset$$ and it contains an irreducible component of dimension $2b+\h^0\mathcal{E}-\h^0\mathcal{E}(-1)-1$.
        \end{theorem}     
        \begin{proof}
        By Lemma \ref{lemma_chib}, $\xi_b=-S+2bH$ defines the non empty wall $W^{\xi_b}$ of type $(c_1,c_2)$ and satisfies the $C-$condition. Thus, we can consider the family $E_{\xi_b}(c_1,c_2)$ of vector bundles given by a non trivial extension of type 
        \begin{equation}
        \label{echi}
         0\rightarrow\mathcal{O}_X(bH)\rightarrow E\rightarrow I_{[SH]}(S-bH)\rightarrow0.   
        \end{equation} Moreover, by Corollary \ref{cor_inclusion}, $E_{\xi_b}(c_1,c_2)$ is an irreducible component of the moduli space $M_{\mathcal{C}_b}(2;c_1,c_2)$. Since any $E\in E_{\xi}(c_1,c_2)$ is given by an extension of type (\ref{echi}), we get $$\dim{E_{\xi}}(c_1,c_2)=\ext^1 (I_Z(S-bH),\mathcal{O}_X(bH))+\dim\mathcal{L}-1,$$ where $\mathcal{L}$ parametrizes the codimension $2$ complete intersections of type $[Z]=SH$. 

Let us first compute $\ext^1 (I_{Z}(S-bH),\mathcal{O}_X(bH))$. Notice that $$\ext^1 (I_{Z}(S-bH),\mathcal{O}_X(bH))=\ext^2(\mathcal{O}_X(bH), I_Z(S-bH+K_X))=\h^2I_Z(S-2bH+K_X).$$

If we consider the exact sequence 
\begin{equation}
\label{eq_SH}
  0\rightarrow\mathcal{O}_X(-S-H)\rightarrow\mathcal{O}_X(-S)\oplus\mathcal{O}_X(-H)\rightarrow I_Z\rightarrow0  
\end{equation}

 twisted by $\mathcal{O}_X(S-2bH+K_X)$ and we take cohomology, we get 
\begin{equation}
\begin{array}{ll}
\label{suc_coh2}
\cdots &\rightarrow \Ho^2\mathcal{O}_X(-2bH+K_X) \oplus \Ho^2\mathcal{O}_X(S-(2b+1)H+K_X) \rightarrow \Ho^2I_Z(S-2bH+K_X) \rightarrow \\
&\rightarrow \Ho^{3}\mathcal{O}_X(-(2b+1)H+K_X) \rightarrow \Ho^{3}\mathcal{O}_X(-2bH+K_X) \oplus \Ho^{3}\mathcal{O}_X(S-(2b+1)H+K_X) \rightarrow \\
&\rightarrow \Ho^3I_Z(S-2bH+K_X) \rightarrow 0\\
\end{array}
\end{equation}

Notice that we have
\[
\begin{array}{llll}
    \h^2\mathcal{O}_X(-2bH+K_X) &= \h^1\mathcal{O}_X(2bH) &= \h^1\mathcal{O}_{\mathbb{P}^2}(2b) &= 0, \\
    \h^{3}\mathcal{O}_X(-(2b+1)H+K_X) &= \h^0\mathcal{O}_X((2b+1)H) &= \h^0\mathcal{O}_{\mathbb{P}^2}(2b+1) &= \binom{2b+2}{2}, \\
    \h^{3}\mathcal{O}_X(-2bH+K_X) &= \h^0\mathcal{O}_X(2bH) &= \h^0\mathcal{O}_{\mathbb{P}^2}(2b) &= \binom{2b+1}{2}, \\
    \h^2\mathcal{O}_X(S-(2b+1)H+K_X) &= 0, & & \\
    \h^{3}\mathcal{O}_X(S-(2b+1)H+K_X) &=0,&  & \\
    \h^{3}I_Z(S-2bH+K_X) &=\h^{3}\mathcal{O}_X(S-2bH+K_X)&= 0.  & \\
\end{array}
\]

Therefore, by (\ref{suc_coh2}) we obtain 

$\begin{array}{l}
\label{suc_coh}
0 \rightarrow \Ho^2I_Z(S-2bH+K_X) 
\rightarrow \Ho^{3}\mathcal{O}_X(-(2b+1)H+K_X) \rightarrow \Ho^{3}\mathcal{O}_X(-2bH+K_X) \rightarrow 0\\
\end{array}$

and thus 
\[
\begin{array}{rcl}
\h^2I_Z(S-2bH+K_X) &=& \h^{3}\mathcal{O}_X(-(2b+1)H+K_X) - \h^{3}\mathcal{O}_X(-2bH+K_X) \\
&=& \binom{2b+2}{2} - \binom{2b+1}{2} \\
&=& 2b+1.
\end{array}
\]

Hence,   $E_{\xi_b}(c_1,c_2)\neq\emptyset$, which implies   $M_{\mathcal{C}}(2;c_1,c_2)\neq\emptyset$.
Finally, again using (\ref{eq_SH}), 
we get 
\[
\begin{array}{lll}
\dim\mathcal{L} &= \dim(\Hom(\mathcal{O}_X(-S-H),\mathcal{O}_X(-S)\oplus\mathcal{O}_X(-H)))-\dim\Aut(\mathcal{O}_X(-S-H))\\
&-\dim\Aut(\mathcal{O}_X(-S)\oplus \mathcal{O}_X(-H))+\dim I_f \\
&=  \h^0\mathcal{E}-\h^0\mathcal{E}(-1)-1,
\end{array}
\]
where $f\in\Hom(\mathcal{O}_X(-S-H),\mathcal{O}_X(-S)\oplus\mathcal{O}_X(-H))$ is a general element and $I_f$ denotes its isotropy group under the action of $\Aut(\mathcal{O}_X(-S-H))\times\Aut(\mathcal{O}_X(-S)\oplus\mathcal{O}_X(-H))$. Putting altogether, we obtain

$\begin{array}{rrl}
 \dim E_{\xi_b}&=&\ext^1 (I_Z(S-bH),\mathcal{O}_X(bH))+\dim\mathcal{L}-1    \\ &=&2b+1+\h^0\mathcal{E}-\h^0\mathcal{E}(-1)-1-1\\ 
     & =& 2b+\h^0\mathcal{E}-\h^0\mathcal{E}(-1)-1.
\end{array}$
 \end{proof}

The above theorem  also allows us to study the Brill-Noether locus of the moduli space $M_{\mathcal{C}_b}(2;S,(b+1)SH-b^2H^2)$. Recall that, in general,  the Brill-Noether locus $W_L^k(2;c_1,c_2)$ of the moduli space $ M_L(2;c_1,c_2)$ is the subvariety which support is the set of $L$-stable bundles $E\in M_L(2;c_1,c_2)$ such that $\h^0E\geq k$.  
  The results in \cite{higherBN} allow us to guarantee the existence of these loci for  varieties of arbitrary dimension. 

Following the notation in \cite{nuestro_ruled}, if $\mathcal{C}$ is a chamber, we denote by $W_\mathcal{C}^k(2;c_1,c_2)$ the Brill-Noether locus $W_L^k(2;c_1,c_2)$ for some $L\in \mathcal{C}$.

\begin{theorem}
\label{BN}
    Let $X=\mathbb{P}(\mathcal{E})\rightarrow\mathbb{P}^2$ be a ruled $3$-fold over $\mathbb{P}^2$ with $c'_i=c_i(\mathcal{E})$.  Let us fix $c_1'+1\leq b\in\mathbb{Z}$ and Chern classes $c_1=S$ and $c_2=(b+1)SH-b^2H^2$. 
    Then, $$W_{\mathcal{C}_b}^k:=W_{\mathcal{C}_b}^k(2;c_1,c_2)\neq\emptyset$$ for $k$ in range $1\leq k\leq d:=\binom{b+2}{2}$.
\end{theorem}

\begin{proof} 
  First of all, since $W_{\mathcal{C}_b}^{k}\supseteq W_{\mathcal{C}_b}^d$ for  $1\leq k\leq d$, it is enough to prove the non emptiness of $W_{\mathcal{C}_b}^d(2;c_1,c_2)$.
  
By Theorem \ref{th_irreducible_comp}, the non empty family $ E_{\xi_b}(c_1,c_2)$ consists of $\mathcal{C}_b-$stable bundles given by a non trivial extension of type $$0\rightarrow \mathcal{O}_X(bH)\rightarrow E\rightarrow I_{[SH]}(S-bH)\rightarrow0.$$ Taking cohomology, we have $\Ho^0\mathcal{O}_X(bH)\subseteq \Ho^0E$, which implies $$\h^0E\geq \h^0\mathcal{O}_X(bH)=\h^0\mathcal{O}_{\mathbb{P}^2}(b)=\binom{b+2}{2}$$ and hence $E\in W_{\mathcal{C}_b}^k$.
\end{proof}

Let us now obtain a decomposition of $M_{\mathcal{C}_b}(2;S, (b+1)SH-b^2H^2)$ in terms of the moduli $M_{\mathcal{C}'_b}(2;S, (b+1)SH-b^2H^2)$ and the family $E_{\xi_b}(S, (b+1)SH-b^2H^2)$, being $\mathcal{C}'_b$  the chamber sharing the wall $W^{\xi_b}$ with $\mathcal{C}_b$.
Before stating our results, let us study the chambers separated by the wall $W^{\xi_b}$.

\begin{proposition}
\label{limit_wall}
 Let  $c_1'+1\leq b\in\mathbb{Z}$ and Chern classes $c_1=S$ and $c_2=(b+1)SH-b^2H^2$.     
 Then:
 \begin{itemize}
     \item[1)]
 There is no wall  of type $(c_1,c_2)$ under $W^{\xi_b}$ and above the non empty wall of type $(c_1,c_2)$ defined by $\xi_{b+1}$.
     \item[2)] 
 There is no wall  of type $(c_1,c_2)$ between $W^{\xi_b}$ and $W^{\xi_{b-1}}$, where $$W^{\xi_{b-1}}=\{L\in(\Num(X)\otimes\mathbb{R})\cap\mathcal{C}_X|\thickspace \xi_{b-1}\cdot L^2=0\}.$$ 
 \end{itemize}
 
\end{proposition}

\begin{proof}
1) First of all, since
     $\xi_{b+1}+S=2(b+1)H$, 
         $$[Z]=\frac{{\xi^2_{b+1}}-S^2}{4}+(b+1)SH-b^2H^2=2(b+1)H^2$$ is a locally complete intersection.
In addition,  $$L= S+\left(2(b+1)-c_1'+\sqrt{2(b+1)\cdot[2(b+1)-c_1']+c_2'}\right)H\in (\Num(X)\otimes\mathbb{R})\cap \mathcal{C}_X$$ satisfies $\xi_{b+1}\cdot L^2=0$. Hence $\xi_{b+1}$ defines a non empty wall $W^{\xi_{b+1}}$ of type $(c_1,c_2)$.

    Moreover, if $L\in W^{\xi_b}$, we have that $0= \xi_{b}\cdot L^2=(\xi_{b+1}-2H)\cdot L^2$. Thus, $\xi_{b+1}\cdot L^2=2H\cdot L^2>0$, which implies that the wall $W^{\xi_{b+1}}$ is under the wall $W^{\xi_{b}}$.

   Let us now prove that there is no wall of type $(c_1,c_2)$ under $W^{\xi_b}$ and above $W^{\xi_{b+1}}$.
   To this end, we will study the possible cases of walls $W^{\eta}$ of type $(c_1,c_2)$.
   Notice that we can avoid the case when $\eta$ or $-\eta$ is effective, since it would not define a wall of type $(c_1,c_2)$.

   \noindent\textbf{Case 1:}
   Let us first study the case $\eta_p= -S+2pH$ with $p>b+1$.
   If $L\in W^{\eta_p}$, then $$\xi_{b+1}\cdot L^2=(\eta_p+2(b+1-p)H)\cdot L^2=2(b+1-p)H\cdot L^2<0,$$
   where the inequality follows from the fact that $H$ is effective and $b+1-p<0$.
    Therefore the wall $W^{\eta_p}$ lies under the wall $W^{\xi_{b+1}}$.

   \noindent\textbf{Case 2:}
   Let us now study the case $\eta_p= -S+2pH$ with $0<p<b$.
   If $L\in W^{\eta_p}$, then $$\xi_b\cdot L^2=(\eta_p+2(b-p)H)\cdot L^2=2(b-p)H\cdot L^2>0,$$

  \noindent where the inequality follows from the fact that $H$ is effective and $p<b$.
    Therefore the wall $W^\eta$ is  above the wall $W^{\xi_{b}}$.

\noindent\textbf{Case 3:}
   Let us now study the family $\eta_{p}= -(2p+1)S+2bH$ with $p\geq1$.

   If $L\in W^{\eta_p}$, then $$\xi_b\cdot L^2=(\eta_p+2S)\cdot L^2= 2pS\cdot L^2>0,$$
\noindent   where the inequality follows from the fact that $S$ is effective and $p\geq1$.
    Therefore the wall $W^{\eta_p}$ is  above the wall $W^{\xi_{b}}$.

\textbf{Case 4: }
   Let us now study the family $\eta_{(a,p)}= -(2a+1)S+2pH$ with $a\geq1$ and $0<p<b$. Then, by Cases 2 and 3,  the wall $\eta_{(a,p)}$ is above the wall $W^{\xi_b}$.
   
\textbf{Case 5:}
Let us now study the family of classes $\eta_{(a,p)}= -(2a+1)S+2pH$ with $a\geq1$ and $p>b$.
Let $L$ be such that $\eta_{(a,p)}\cdot L^2=0$ and assume that $L$ is between the walls $W^{\xi_b}$ and $W^{\xi_{b+1}}$, that is, $$\xi_b\cdot L^2<0<\xi_{b+1}\cdot L^2.$$

On one hand, since $\xi_b\cdot L^2<0=\eta_{(a,p)}\cdot L^2,$
we get 
\begin{equation}
\label{SL1}
S\cdot L^2<\frac{p-b}{a}H\cdot L^2.    
\end{equation}
On the other hand, since $\xi_{b+1}\cdot L^2>0=\eta_{(a,p)}\cdot L^2,$
we get 
\begin{equation}
    \label{SL2}
    S\cdot L^2>\frac{p-b-1}{a}H\cdot L^2.
\end{equation}
Finally, since $\eta_{(a,p)}\cdot L^2=0$ is equivalent to $S\cdot L^2=\frac{2p}{2a+1}H\cdot L^2$, by (\ref{SL1}) and (\ref{SL2}) we get $$\frac{p-b-1}{a}<\frac{2p}{2a+1}<\frac{p-b}{a},$$
which is equivalent to $$(2a+1)b<p<(2a+1)(b+1).$$

Therefore, we only have to study the family of $\eta_{(a,p)}=-(2a+1)S+2pH$ with $a\geq1$ and $(2a+1)b<p<(2a+1)(b+1).$

Let us prove that $\eta_{(a,p)}$ does not define a wall of type $(c_1,c_2)$.
In fact, if $\eta_{(a,p)}$ defines a wall of type $(c_1,c_2)$, then  $$[Z]=\frac{\eta_{(a,p)}^2-S^2}{4}+(b+1)SH-b^2H^2$$ is a locally complete interesection and thus $[Z]\cdot L\geq0$ for any $L\in\mathcal{C}_X$. In particular, if $L_0=S+\beta H$ with $\beta\gg0$, we have that $[Z]\cdot L_0\geq0$, which implies that $$a(a+1)c_1'+b+1-(2a+1)p\geq0.$$
Since, $p>(2a+1)b$, we get $$a(a+1)c_1'+b+1-(2a+1)^2b=a(a+1)(c_1'-4b)+1\geq0,$$ which is a contradiction since $a\geq1$ and $c_1'-4b\leq-1$.

2)The proof follows step by step the proof of  1) and we omit details.

\end{proof}

Let us now prove that we can always find polarizations in the chambers with common wall $W^{\xi_b}$. By Proposition \ref{limit_wall}, in order to get polarizations in $\mathcal{C}'_b$ and $\mathcal{C}_b$, it is enough to get polarizations between $W^{\xi_{b-1}}$ and $W^{\xi_b}$, and between $W^{\xi_{b}}$ and $W^{\xi_{b+1}}$, respectively.

\begin{proposition}
Let $X=\mathbb{P}(\mathcal{E})\rightarrow\mathbb{P}^2$ be a ruled $3$-fold over $\mathbb{P}^2$ with $c'_i=c_i(\mathcal{E})$.  Let $c_1+1\leq b\in\mathbb{Z}$ and Chern classes $c_1=S$ and $c_2=(b+1)SH-b^2H^2$.
    Then, $\Num(X)\cap \mathcal{C}_b\neq\emptyset$ and $\Num(X)\cap \mathcal{C}'_b\neq\emptyset$. 
    
\end{proposition}
\begin{proof}
    Let us first prove that $\Num(X)\cap \mathcal{C}'_b\neq\emptyset$. To this end, by Proposition \ref{limit_wall}, it is enogh to find a polarization $L$ verifying $\xi_{b-1}\cdot L^2<0<\xi_b\cdot L^2$.
    Consider $L=S+\beta H\in \Num(X)\cap\mathcal{C}_X$ 
    and define
   
    $\begin{array}{lll}
        f(\beta) :=&\xi_b\cdot L^2&=-\beta^2+(2b-c_1')\beta+2bc_1'-[(c_1')^2-c_2']  \\
        g(\beta):=& \xi_{b-1}\cdot L^2&=-\beta^2+[2(b-1)-c_1']\beta+2(b-1)c_1'-[(c_1')^2-c_2'] .
    \end{array}$

Notice that $g(\beta)=f(\beta)-2\beta-2c_1'$. If we denote by $A_1:=2b-c_1'$ and $A_0:=2bc_1'-[(c_1')^2-c_2']$, then $ f(\beta) =-\beta^2+A_1\beta+A_0$ and $g(\beta)=f(\beta)-2\beta-2c_1' $
    are two parabolas in the variable $\beta\in\mathbb{R}^+$.
Let us prove that there exists an integer $\beta_0>0$  such that $f(\beta_0)>0>g(\beta_0)$.
Notice that $g$ has the vertex on the point $$P=(\frac{A_1-2}{2},\frac{(A_1-2)^2}{4} +A_0-2c_1').$$ Since $b\geq c_1'+1$, $P$  belongs to $\mathbb{R}^+\times\mathbb{R}^+$, which implies that there exists a positive $y_0$ such that $g(y_0)=0$. 
Then, since $g$ is continuous an decreases in $[y_0,+\infty)$, we have $g(y_0+1)<0$.

Let us now check that $f(y_0+1)>0$.
Notice that $$f(y_0+1)=g(y_0+1)+2(y_0+1)+2c_1'=g(y_0)-1+A_1+2c_1'=2b+c_1'-1>0.$$
Since $f$ is continuos and $f(\beta)>g(\beta)$ for any $\beta$, we have that $f(\beta)>0$ for $\beta\in[y_0,y_0+1]$.
Putting altogether $f(\beta)>0>g(\beta)$ for any $\beta\in (y_0,y_0+1]$.
If $y_0$ is an integer, take $\beta_0=y_0+1$; otherwise, take $\beta_0=\lceil y_0\rceil\in(y_0,y_0+1) $.
Hence we have proved that $L=S+\beta_0H$ is a polarization in $\mathcal{C}'_b$.

Following the same argument, considering  $f(\beta):=\xi_{b+1}\cdot L^2$ and $g(\beta):=\xi_{b}\cdot L^2$ for $L=S+\beta H$, we prove that there exists a polarization in $\mathcal{C}_b$.

\end{proof}

\begin{theorem}
\label{prop_decom} Let $X=\mathbb{P}(\mathcal{E})\rightarrow\mathbb{P}^2$ be a ruled $3$-fold over $\mathbb{P}^2$ with $c'_i=c_i(\mathcal{E})$.  Let  $c_1'+1\leq b\in\mathbb{Z}$ and Chern classes $c_1=S$ and $c_2=(b+1)SH-b^2H^2$.
   Then, $$M_{\mathcal{C}_b}(2;c_1,c_2)=M_{\mathcal{C}'_b}(2;c_1,c_2)\sqcup E_{\xi_b}(c_1,c_2).$$ 
\end{theorem}
\begin{proof}


    First of all, we will see that we have a disjoint union on the righthand side. 
    For any $E\in E_{\xi_b}(c_1,c_2)$,  $\mathcal{O}_X(bH)\hookrightarrow E$. On the other hand, for $L'\in \mathcal{C}'_b$, $\xi_b\cdot (L')^2>0$, which is equivalent to $(bH)\cdot (L')^2>\frac{c_1\cdot L^2}{2}$. Hence, $E$ is $L'$-unstable and then $E\not\in M_{\mathcal{C}'_b}(2;c_1,c_2)$.
    
      By Corollary \ref{cor_inclusion}, $E_{\xi_b}( c_1,c_2)\hookrightarrow M_{\mathcal{C}_b}(2;  c_1,c_2)$.
Let us now prove that $M_{\mathcal{C}'_b}(2;c_1,c_2)\subseteq M_{\mathcal{C}_b}(2;c_1,c_2).$
Let us take $E\in M_{\mathcal{C}'_b}(2;c_1,c_2)$.
  If $E\not \in M_{\mathcal{C}_b}(2;c_1,c_2)$, then $E$ is $L'$-stable for any $L'\in \mathcal{C}'_b$ but $L$-unstable for some $L\in \mathcal{C}_b$. 
  Hence, by Lemma \ref{lemma_separar}, there exists $\xi=2G-c_1\in\Num(X)$ for some $\mathcal{O}_X(G)\hookrightarrow E$ defining a non empty wall $W^{\xi}$ of type $(c_1,c_2)$ that separates $L\in\mathcal{C}_b$ and $L'\in \mathcal{C}'_b$. Since $\Closure(\mathcal{C}_b)\cap W^{\xi_b}\neq \emptyset$ and $\Closure(\mathcal{C}'_b)\cap W^{\xi_b}\neq \emptyset$, we get $W^{\xi_b}=W^{\xi}$.
    Finally, by Lemma \ref{lemma_chib} ii) and the fact that $\xi\cdot {L'}^2<0$, we get $\xi=-\xi_b$ and thus $E(-G)$ is a vector bundle that sits in the exact sequence $$0\rightarrow\mathcal{O}_X\rightarrow E(-G)\rightarrow I_Z(\xi_b)\rightarrow0.$$ This implies that $\Ho^0\left({\mathcal{N}_{Z/X}}\otimes\mathcal{O}_X(-\xi_b)_{|Z}\right)\neq0$, which is a contradiction since by the proof of Lemma \ref{lemma_chib} (ii) this cohomology group vanishes. Hence $E\in M_{\mathcal{C}_b}(2;c_1,c_2)$, which proves $M_{\mathcal{C}_b'}(2;c_1,c_2)\subseteq M_{\mathcal{C}_b}(2;c_1,c_2)$.
   
Let us now prove the other  inclusion. We follow the same steps as before. If we assume that $E\in M_{\mathcal{C}_b}(2;c_1,c_2)$ is $\mathcal{C}_b'$-unstable, we get that $E$
sits in the exact sequence $$0\rightarrow\mathcal{O}_X(G)\rightarrow E\rightarrow I_Z(S-G)\rightarrow0,$$ where $[Z]=\frac{\xi^2-{c_1}^2}{4}+c_2$ for $\xi=2G-c_1$. 
Moreover, we can check that $\xi$ defines a non-empty wall of type $(c_1,c_2)$ separating the chambers $\mathcal{C}_b$ and $\mathcal{C}'_b$. Thus, since $\mathcal{C}_b$ and $\mathcal{C}'_b$ are only separated by the wall $W^{\xi_b}$,  we conclude by Lemma \ref{lemma_chib} that $\xi=\xi_b$.
 Finally, since $\xi_b$ defines the non empty wall $W^{\xi_b}$ and satisfies the $C-$condition, we get by definition that $E\in E_{\xi_b}(c_1,c_2)$
 and thus $$M_{\mathcal{C}_b}(2;c_1,c_2)\backslash M_{\mathcal{C}'_b}(2;c_1,c_2) \subseteq E_{\xi_b}(c_1,c_2),$$ which is equivalent to $M_{\mathcal{C}_b}(2;c_1,c_2)\subseteq M_{\mathcal{C}'_b}(2;c_1,c_2) \sqcup E_{\xi_b}(c_1,c_2).$

\noindent Finally, both inclusions imply $M_{\mathcal{C}_b}(2;c_1,c_2)= M_{\mathcal{C}'_b}(2;c_1,c_2) \sqcup E_{\xi_b}(c_1,c_2).$
\end{proof}

\begin{remark}
If we assume that there are polarizations in the face $\mathcal{F}_b\subseteq W^{\xi_b}$, this is $\mathcal{F}_b\cap\Num(X)\neq\emptyset$, then  we get $$M_{\mathcal{C}_b}(2;c_1,c_2)=M_{\mathcal{F}_b}(2;c_1,c_2)\sqcup E_{\xi_b}(c_1,c_2).$$
Notice that, if  $E\in M_{\mathcal{C}_b}(2;c_1,c_2)\backslash  M_{\mathcal{F}_b}(2;c_1,c_2)$, then by Remark \ref{M_cara} $E$ is $\mathcal{C}'_b$-unstable. Therefore, following analogues arguments as in the proof of Theorem \ref{prop_decom}, we get $E\in E_{\xi_b}(c_1,c_2)$. The fact that $M_{\mathcal{F}_b}(2;c_1,c_2)\cap E_{\xi_b}(c_1,c_2)=\emptyset$ and the  inclusion $M_{\mathcal{F}_b}(2;c_1,c_2)\sqcup E_{\xi_b}(c_1,c_2)\subseteq M_{\mathcal{C}_b}(2;c_1,c_2)$ follow from Remark \ref{M_cara} and  Corollary \ref{cor_inclusion}.
\end{remark}

Let us now see that if we cross the wall $W^{\xi_{b+1}}$ we get the emptiness of the moduli space.

\begin{theorem}
\label{th_vacio1}
    Let $X=\mathbb{P}(\mathcal{E})\rightarrow\mathbb{P}^2$ be a ruled $3$-fold  over $\mathbb{P}^2$ with $c'_i=c_i(\mathcal{E})$.  Let  $c_1'+1\leq b\in\mathbb{Z}$ and Chern classes $c_1=S$ and $c_2=(b+1)SH-b^2H^2$.  If $L$ is a polarization lying on $W^{\xi_{b+1}}$ or under this wall, then $M_{L}(2;c_1,c_2)=\emptyset$.
\end{theorem}

\begin{proof}

First of all, notice that, for $L$  lying on $W^{\xi_{b+1}}$ or under this wall, we have $\xi_{b+1}\cdot L^2\leq0$, which is equivalent to 
\begin{equation}
\label{cota_debajo}
    2(b+1) H\cdot L^2\leq S\cdot L^2.
\end{equation}

    Let us assume that $M_{L}(2;c_1,c_2)\neq \emptyset$ and take $E\in M_{L}(2;c_1,c_2)$. 
If $F\cong\mathbb{P}^1$ is a general fiber of $\pi:\mathbb{P}(\mathcal{E})\rightarrow\mathbb{P}^2$, we get $E_{|F}\cong \mathcal{O}_{\mathbb{P}^1}(a)\oplus \mathcal{O}_{\mathbb{P}^1}(1-a)$, with $a>0$ some integer. Thus $\mathcal{O}_{\mathbb{P}^1}(a)\hookrightarrow E_{|F}$, which implies $\mathcal{O}_X(aS+tH)\hookrightarrow E$ for some $t\in\mathbb{Z}$. Hence, for some integers $l\geq1$ and $m$, we can find a section $s\in\Ho^0E(-lS-mH)$, which scheme of zeros has codimension $2$.
    Since $E$ is $L$-stable, $\mathcal{O}_X(lS+mH)\hookrightarrow E$ implies $(lS+mH)\cdot L^2<\frac{S\cdot L^2}{2}$, which is equivalent to 
    \begin{equation}
    \label{cota_lstable}
        2m H\cdot L^2<-(2l-1)S\cdot L^2.
    \end{equation}
By (\ref{cota_debajo}), $$-S\cdot L^2\leq -2(b+1)H\cdot L^2$$ and, since $l\geq1$ and $S\cdot L^2>0$, $$-(2l-1)S\cdot L^2\leq -2(2l-1)(b+1)H\cdot L^2,$$ which together with (\ref{cota_lstable}) implies $$2m H\cdot L^2<-(2l-1)S\cdot L^2\leq -2(2l-1)(b+1)H\cdot L^2.$$ Hence, since $H\cdot L^2>0$, we get   $$m<-(2l-1)(b+1).$$   Moreover, since $s\in\Ho^0E(-lS-mH)$  has an scheme of zeros of codimension $2$, $$c_2(E(-lS-mH))= (b+1+(2l-1)m)SH+(m^2-b^2)H^2+l(l-1)S^2$$ is a codimension $2$ locally complete intersection. In particular, if we consider $L_k=2S+k H$ with $k\gg0$, we have $c_2(E(-lS-mH))\cdot L_k\geq0$, which implies
    \begin{equation}
    \label{codim2}
        b+1+(2l-1)m+c_1'l(l-1)\geq0.
    \end{equation}
Finally, since $m<-(2l-1)(b+1)$, by (\ref{codim2}) we get $l(l-1)(c_1'-4b-4)\geq0$, a contradiction. Hence $M_{L}(2;c_1,c_2)=\emptyset$.
\end{proof}

We can also study the moduli space when considering  $c_2=bSH-b^2H^2$.
Notice that, for this case, since $[Z]=\frac{\xi_b^2-{c_1^2}}{4}+c_2=0$, $\xi_b=-S+2bH$ also defines a non empty wall $W^{\xi_b}$ of type $(S,bSH-b^2H^2)$ (the rest follows from Lemma \ref{lemma_chib} i)).

\begin{theorem}
       Let $X=\mathbb{P}(\mathcal{E})\rightarrow\mathbb{P}^2$ be a ruled $3$-fold over $\mathbb{P}^2$ with $c'_i=c_i(\mathcal{E})$.  Let  $c_1'+1\leq b\in\mathbb{Z}$ and Chern classes $c_1=S$ and $c_2=bSH-b^2H^2$.
   If $L$ is a polarization lying on the wall $W^{\xi_b}$ of type $(c_1,c_2)$ or below it, then $M_L(2;c_1,c_2)=\emptyset.$
\end{theorem}
\begin{proof}
The proof is analogous to the one in Theorem \ref{th_vacio1}.
By assumption, $\xi_{b}\cdot L^2\leq0$, which is equivalent to 
\begin{equation}
    2b H\cdot L^2\leq S\cdot L^2.
\end{equation}

    Let us assume that $M_{L}(2;c_1,c_2)\neq \emptyset$ and take $E\in M_{L}(2;c_1,c_2)$. 
Since $c_1(E)=S$, for $l\geq1$ and $m$ some integers, we can find a section $s\in\Ho^0E(-lS-mH)$ which scheme of zeros has codimension $2$.
    Since $E$ is $L-$stable, $\mathcal{O}_X(lS+mH)\hookrightarrow E$ implies $(lS+mH)\cdot L^2<\frac{S\cdot L^2}{2}$, which is equivalent to 
    \begin{equation}
        2m H\cdot L^2<-(2l-1)S\cdot L^2.
    \end{equation}

Therefore we get
$m<-(2l-1)b$. Moreover, since $s\in\Ho^0E(-lS-mH)$  has a codimension $2$ scheme of zeros, $$c_2(E(-lS-mH))= (b+(2l-1)m)SH+(m^2-b^2)H^2+l(l-1)S^2$$ is a codimension $2$ locally complete intersection. In particular, if we consider $L_k=2S+k H$ with $k\gg0$, we have $c_2(E(-lS-mH))\cdot L_k\geq0$, which implies
    \begin{equation}
    \label{codim2'}
        b+(2l-1)m+c_1'l(l-1)\geq0.
    \end{equation}
Finally, since $m<-(2l-1)b$, by (\ref{codim2'}) we get $l(l-1)(c_1'-4b)\geq0$, a contradiction. Hence $M_{L}(2;c_1,c_2)=\emptyset$.

\end{proof}

\vspace{3mm}
\noindent
{\bf Conflict of interest:} On behalf of all authors, the corresponding author states that there is no conflict of interest.


\begin{thebibliography}{99}
\bibitem{Arrondo} E. Arrondo, \textit{A home-made Hartshorne-Serre correspondence},  Rev. Mat. Complut. 20, No. 2, 423--443 (2007)


\bibitem{Coanda}
I. Coandă,
\textit{Stable rank 3 vector bundles on \(\mathbb{P}^3\) with \(c_1 = 0\), \(c_2 = 3\)}, preprint: \url{https://arxiv.org/abs/2103.11723}, (2021).

\bibitem{nuestro_ruled}L. Costa and I. Macías Tarrío, \textit{Brill-Noether theory of stable vector bundles on ruled surfaces}, Mediterr. J. Math. 21, No. 3, Paper No. 118, 22 p. (2024)

\bibitem{trabajo_futuro}L. Costa and I. Macías Tarrío, \textit{Moduli spaces of stable bundles on some $3$-folds}, preprint (2025)

\bibitem{Brill-Noether}L. Costa, R. M. Miró-Roig, \textit{Brill-Noether Theory for moduli spaces of sheaves on algebraic varieties}, Forum Math. 22, No. 3, 411--432 (2010).

\bibitem{Grauert} H. Grauert, G. Mülich, \textit{Vektorbündel vom Rang 2 über dem n-dimensionalen komplexprojektiven Raum}, Manusc. Math., 16 75-100 (1975). 

\bibitem{Hartshorne1978}R. Hartshorne,
\textit{Stable vector bundles of rank 2 on $\mathbb{P}^3$}, Mathematische Annalen, 238 229--280 (1978).

\bibitem{higherBN} B. Nugent, \textit{Higher Dimensional Brill-Noether Loci and Moduli for Very Ample Line Bundles
}, preprint: https://arxiv.org/abs/2405.17689

\bibitem{Qin1}Z. Qin, \textit{Equivalence classes of polarizations}, J. Differ. Geom. 37, No. 2, 397--415 (1993)

\bibitem{Qin2}Z. Qin, \textit{Moduli spaces of stable rank-2 bundles on ruled surfaces}, Invent. Math. 110, No. 3, 615--626 (1992)

\bibitem{Vitter2003}A. Vitter,
\textit{On stable bundles of ranks 2 and 3 on $\mathbb{P}^3$}, preprint:
\url{https://arxiv.org/abs/math/0310073} (2003).

\bibitem{Vogelaar} J. A. Vogelaar, \textit{Constructing vector bundles from codimension-two subvarieties}, PhD
thesis. Leiden (1978).


\end{thebibliography}
\end{document}